\documentclass{amsart}

\usepackage{amsmath,amssymb,amsthm}
\usepackage{hyperref}
\usepackage{xcolor}

\usepackage[small,nohug,heads=vee,dpi=600]{diagrams}
\diagramstyle[labelstyle=\scriptstyle]
\newarrow {Dashto} {}{dash}{}{dash}>
\newarrow {Dash} {}{dash}{}{dash}{}
\newarrow {Doubleto} {}{=}{}{=}{=>}

\newtheorem{thmintr}{Theorem}
\newtheorem{prop}{Proposition}[section]
\newtheorem{lem}[prop]{Lemma}

\theoremstyle{definition}

\newtheorem{rem}[prop]{Remark}

\newtheorem*{ack}{Acknowledgement}

\def\co{\colon\thinspace}

\newcommand{\C}{\mathbb C}

\newcommand{\rmd}{\mathrm d}
\newcommand{\D}{\mathbb D}

\newcommand{\rme}{\mathrm e}

\newcommand{\bfh}{\mathbf h}

\newcommand{\rmi}{\mathrm i}

\newcommand{\bfp}{\mathbf p}

\newcommand{\bfq}{\mathbf q}

\newcommand{\R}{\mathbb R}

\newcommand{\bfs}{\mathbf s}

\newcommand{\bft}{\mathbf t}

\newcommand{\WW}{\mathcal W}
\newcommand{\bfw}{\mathbf w}

\newcommand{\bfx}{\mathbf x}

\newcommand{\bfy}{\mathbf y}

\newcommand{\Z}{\mathbb Z}

\newcommand{\bfz}{\mathbf z}

\newcommand{\lra}{\longrightarrow}
\newcommand{\ra}{\rightarrow}

\DeclareMathOperator{\ev}{\mathrm{ev}}

\DeclareMathOperator{\Int}{\mathrm{Int}}

\begin{document}

\author{Kilian Barth}
\address{Fraunhofer-Institut f\"ur Hochfrequenzphysik und Radartechnik FHR
Fraunhoferstr. 20, D-53343 Wachtberg, Germany}
\email{kilian.barth@fhr.fraunhofer.de}
\author{Jay Schneider}
\address{Mathematisches Institut, Westf\"alische Wilhelms-Universit\"at M\"unster,
Einsteinstr. 62, D-48149 M\"unster, Germany}
\email{jay.schneider@uni-muenster.de}
\author{Kai Zehmisch}
\address{Mathematisches Institut, Justus-Liebig-Universit\"at Gie{\ss}en,
Arndtstra{\ss}e 2, D-35392 Gie{\ss}en, Germany}
\email{Kai.Zehmisch@math.uni-giessen.de}

\title[Symplectic dynamics of contact isotropic torus complements]{Symplectic dynamics of contact isotropic torus complements}

\date{13.09.2018}

\begin{abstract}
  We determine the homotopy type
  of isotropic torus complements
  in closed contact manifolds
  in terms of Reeb dynamics
  of special contact forms.
  For that we utilise holomorphic curve techniques
  known from symplectic field theory
  as Gromov--Hofer compactness
  and localised transversality
  on non-compact contact manifolds.
\end{abstract}

\thanks{This research is part of projects in the
SFB 878 {\it Groups, Geometry and Actions}
and the SFB/TRR 191
{\it Symplectic Structures in Geometry, Algebra and Dynamics},
both funded by the DFG}

\maketitle


\section{Introduction\label{sec:intro}}

By the isotropic neighbourhood theorem
a neighbourhood of a closed isotropic submanifold $Q$
in a given contact manifold is determined
by the diffeomorphism type of $Q$
and by the isomorphism class of the conformally
symplectic normal bundle $\mathrm{CSN}(Q)$ of $Q$,
cf.\ \cite[Theorem 2.5.8]{gei08}.
For instance if $\mathrm{CSN}(Q)$ is trivial,
a trivialisation of $\mathrm{CSN}(Q)$ determines
a local model given by
a neighbourhood of $Q$ in the contactisation
of $T^*Q\times\C^{n-d}$,
$d=\dim Q\leq n$,
cf.\ \cite[Section 3.1]{vkoe18}.
On the other hand,
the restriction of any defining contact form
to the tangent bundle of a compact hypersurface
determines the germ of the contact structure,
see Ding--Geiges \cite[Proposition 6.4]{dg12}.
In particular,
this applies to the boundary
of a disc-like neighbourhood of $Q$.
Combined with local contact inversion
(see Proposition \ref{prop:contact-inversion})
it turns out that there is no canonical distinction
between in- or outside for this hypersurface
if $\mathrm{CSN}(Q)$ is trivial.

In this work we consider the case where $Q$ is a torus $T^d$
and where the complement $M$ of a tubular neighbourhood of $Q$
is compact.
Assuming $n>d$
we investigate to which extent a choice of a contact form
on $M$ that is of model type near the boundary
determines the topology of $M$.
As demonstrated by Eliashberg--Hofer \cite{eh94}
if $\dim M=3$ and
Geiges--Zehmisch \cite{gz16b} for $\dim M\geq5$,
$M$ is diffeomorphic to an $(2n+1)$-dimensional ball
whenever $M$ does not have any
short contractible periodic Reeb orbits
and $\partial M$ is a sphere.
This situation corresponds to $Q=*$.
In fact,
Eliashberg--Hofer \cite{eh94} proved
a global Darboux theorem in the absence of 
short periodic Reeb orbits if $\dim M=3$.
In contrast,
Geiges--R\"ottgen--Zehmisch \cite{grz14}
constructed an aperiodic Reeb flow with trapped orbits
on $\R^{2n+1}$, $n\geq2$,
that is standard outside a compact set.
In order to find a relation between the topology
of the isotropic knot complement $M$
and the existence of short periodic Reeb orbits on $M$
we will utilise holomorphic curves as it is typical
in symplectic dynamics as propagated by
Bramham--Hofer \cite{bh12}.


\subsection{Main result\label{subsec:mainres}}

Let us assume that $Q$ is the $d$-dimensional torus
\[
T^d:=\R^d/2\pi\Z^d\,.
\]
We consider a compact, connected $(2n+1)$-dimensional
strict contact manifold $(M,\alpha)$ with boundary
\[
\partial M=S\big(T^*Q\oplus\underline{\R}^{2n+1-2d}\big)
\]
equal to the unit sphere bundle
of the stabilised cotangent bundle
$T^*Q\oplus\underline{\R}^{2n+1-2d}$ of $T^*Q$.
In particular, the boundary of $M$ is diffeomorphic to
\[
\partial M=T^d\times S^{2n-d}\,.
\]
The aim of this work is to give a criterion for $M$
to be diffeomorphic to the unit disc bundle
\[
D\big(T^*Q\oplus\underline{\R}^{2n+1-2d}\big)
=T^d\times D^{2n+1-d}
\]
in terms of the infimum $\inf_0(\alpha)$ of all positive periods of
{\it contractible} closed Reeb orbits of the Reeb vector field of $\alpha$
and the following {\it embeddability condition}:

Write
\[
Z:=\R\times T^*T^d\times D^2\times\C^{n-1-d}
\]
for the model neighbourhood of an isotropic torus
$T^d$ with trivial conformally symplectic normal bundle
inside a contact manifold
equipped with the contact form
\[
\alpha_Z:=
\rmd b
+\sum_{j=1}^dp_j\rmd q_j
+\frac12\big(x_0\rmd y_0-y_0\rmd x_0\big)
-\sum_{j=1}^{n-1-d}y_j\rmd x_j\,,
\]
where $b\in\R$,
$p_j,q_j$ are coordinates on the cotangent bundle $T^*T^d$,
$x_0,y_0$ are coordinates on the closed unit disc $D^2$,
and $x_j+\rmi y_j$ are coordinates on $\C^{n-1-d}$.
We will use the following short form
\[
\alpha_Z=
\rmd b
+\bfp\!\;\rmd\bfq
+\frac12\big(x_0\rmd y_0-y_0\rmd x_0\big)
-\bfy\rmd\bfx\,,
\]
of the contact form during the text.

We say that $\partial M$ admits a {\bf contact embedding}
into $(Z,\alpha_Z)$ if there exists
a strict contact embedding $\varphi$
of a collar neighbourhood $U$ of $\partial M\subset M$
into the interior of $Z$
in the sense that  $\varphi^*\alpha_Z=\alpha$
such that
\begin{itemize}
\item
each flow line of the Reeb vector field $\partial_b$
intersects $\varphi(\partial M)\subset Z$ in at most two points,
\item
the image $\varphi(U)$ is
contained in the bounded component of
$Z\setminus\varphi(\partial M)$, and
\item
$\varphi(\partial M)$
is smoothly isotopic to
$S\big(T^*T^d\oplus\underline{\R}^{2n+1-2d}\big)$
inside $Z$.
\end{itemize}

\begin{thmintr}
\label{thm:mainthm}
 Let $(M,\alpha)$ be a strict contact manifold
 as described above such that $\partial M$
 has a contact embedding into $(Z,\alpha_Z)$
 with $n>d$.
 If $\inf_0(\alpha)\geq\pi$,
 then $M$ and $T^d\times D^{2n+1-d}$
 are homotopy equivalent if $n=2$
 and diffeomorphic otherwise.
\end{thmintr}

Observe the standing assumption $n\geq d$.
The extra condition we require in the theorem is $n\neq d$
meaning that $Q$ is {\bf subcritically isotropic}.
This assumption allows us the use of holomorphic discs
as a $\C$-factor can be split off in the model situation.

The case $n=1$ is covered by the work of Eliashberg--Hofer \cite{eh94};
the critical case $n=1$ and $d=1$,
without any shortness assumption on $\inf_0(\alpha)$,
is proved by Kegel--Schneider--Zehmisch \cite{kschz}
using a different method.
Therefore, 
we assume $n\geq2$ (and $n>d$) throughout the article.
The case $d=0$ is due to Geiges--Zehmisch \cite{gz16b}
in which even being diffeomorphic can be concluded.

Performing contact connected sum
of $Z$ with any contact manifold
one obtains a periodic Reeb orbit
of period strictly less than
(but arbitrarily close to) $\pi$
contained in the belt sphere,
cf.\ \cite[Remark 1.3.(1)]{gz16b}.
Hence,
the bound $\pi$ in Theorem \ref{thm:mainthm} is optimal.
Moreover,
the contrapositive of Theorem \ref{thm:mainthm}
can be used to prove existence
of periodic Reeb orbits on non-compact manifolds.
Consider a compact contact manifold $(M,\xi)$
whose boundary has precisely two connected components
each admitting a contact embedding into $(Z,\alpha_Z)$
individually.
Using the Reeb flow on the model $(Z,\alpha_Z)$
the images can be assumed to be not nested
so that a gluing of $(M,\xi)$ to $(Z,\alpha_Z)$
along the boundary is possible.
The gluing result can not be homotopy equivalent to
$T^d\times D^{2n+1-d}$
so that any contact form on $(M,\xi)$
standard near the boundary possesses
a contractible periodic Reeb orbit,
cf.\ \cite[Remark 1.3.(4)]{gz16b}.

For existence results of periodic Reeb orbits
on non-compact contact manifolds
with asymptotic and periodic boundary conditions
we refer to the work of
Suhr--Zehmisch \cite{sz16}
and Bae--Wiegand--Zehmisch \cite{bwz},
respectively.


\subsection{Filling by holomorphic discs\label{subsec:filling}}

The basic idea of the proof of Theorem \ref{thm:mainthm}
is the same as described in \cite[Section 1.2]{gz16b}
invoking {\it filling by holomorphic discs}
techniques as worked out in
\cite{kbdiss,gz10,gz13,gz16a,gz16b,zeh16}.
Using the contact embedding $\varphi$ of $\partial M$
into $(Z,\alpha_Z)$ we form a new strict contact manifold
$(\hat{M},\hat{\alpha})$ by replacing the bounded component
of $\R\times T^*T^d\times\C^{n-d}\setminus\varphi(\partial M)$
by $M$.
The contact form $\hat{\alpha}$ equals $\alpha$
on $M$ and
\[
\rmd b
+\bfp\!\;\rmd\bfq
+\frac12\big(x_0\rmd y_0-y_0\rmd x_0\big)
-\bfy\rmd\bfx
\]
on the unbounded component of
$\R\times T^*T^d\times\C^{n-d}\setminus\varphi(\partial M)$.
Observe that the latter contact form
coincides with $\alpha_Z$ on $Z$.
We define $\hat{Z}$ similarly by gluing $M$ into $Z$.

Further, we will consider holomorphic maps
\[
u=(a,f)\co\D\lra W
\]
defined on the closed unit disc $\D\subset\C$
and taking values in the symplectisation
$W$ of $(\hat{M},\hat{\alpha})$
subject to varying Lagrangian boundary conditions.
The moduli space $\WW$ of all such holomorphic discs
carries an evaluation map
\[
 \begin{array}{rccc}
  \ev\co & \WW\times\D     & \longrightarrow & \hat{Z}\\
            & \bigl( (a,f),z\bigr) & \longmapsto     & f(z)\,.
 \end{array}
\]
{\it A priori} $\ev$ takes values in $\hat{M}$,
but we will show that $\ev$ indeed takes values
in the smaller set $\hat{Z}\subset\hat{M}$.
It will turn out that
either the evaluation map $\ev$ is proper and surjective of degree one,
in which case we can draw conclusions with the $s$-cobordism theorem
as in the work of Barth--Geiges--Zehmisch \cite{bgz},
or the moduli space $\WW$ is not locally compact
in the sense that there will be breaking off of finite energy planes.
By a result of Hofer \cite{hof93,hof99}
this in turn results in the existence of short
contractible periodic Reeb orbits of $\alpha$
as the Reeb flow of $\alpha_Z$ is linear,
given by $\partial_b$.

Observe that a contact embedding
of $\partial M$ into $(Z,\alpha_Z)$
yields an embedding of $\partial M$
into $\R\times T^*T^d\times D^2_r\times\C^{n-1-d}$
for some slightly smaller radius $r\in(0,1)$.
The proof of Theorem \ref{thm:mainthm}
that we are going to present in this work
will show that being short for a contractible
periodic Reeb orbit should mean to have period
less than or equal to $\pi r^2$.
Therefore, the above mentioned second alternative
will be excluded by requiring $\inf_0(\alpha)>\pi r^2$
as an Arzel\`a--Ascoli argument shows.
For ease of notation we will assume $r=1$
so that we assume the stronger condition
$\inf_0(\alpha)>\pi$
during the proof.


\subsection{Relevance of the torus\label{subsec:reltorus}}

Large parts of the argument
work under considerably weaker assumptions --
mainly the topological part,
which is similar to \cite{bgz}.
In order to set up
the holomorphic disc analysis
we use a foliation by Lagrangian submanifolds of $T^*Q$
as parametrised boundary condition.
Moreover,
we use a choice of strictly plurisubharmonic potential
for the Liouville form on $T^*Q$
that together with the maximum principle
ensures $C^0$-bounds in the compactness argument,
see Section \ref{subsubsec:maxprinc}.
This together with the Niederkr\"uger map,
which we use to construct holomorphic discs,
works particularly well in global (periodic) coordinates on $T^*Q$.
It is not clear how to change the setup 
to enlarge the class of examples.


\section{Standard holomorphic discs\label{sec:stholdiscs}}

The model contact manifold $(Z,\alpha_Z)$
is the contactisation of the Liouville manifold
\[
(V,\lambda_V):=
\Big(
T^*T^d\times D^2\times\C^{n-1-d},\,
\bfp\!\;\rmd\bfq
+\frac12\big(x_0\rmd y_0-y_0\rmd x_0\big)
-\bfy\rmd\bfx
\Big)
\,,
\]
which contains the holomorphic discs
$\{\bfw\}\times D^2\times\{\bfs+\rmi\bft\}$.
The aim of this section is to describe a lift
of these holomorphic discs to the symplectisation
of $(Z,\alpha_Z)$.
These holomorphic discs will appear
as the standard discs of the moduli space $\WW$
and serve as a description of the end of $\WW$.
In order to lift we proceed in two steps.
The first will be a lift to
$\C\times T^*T^d\times D^2\times\C^{n-1-d}$;
the second is a transformation along a biholomorphic map
$\Phi$ from $\R\times\R\times T^*T^d\times D^2\times\C^{n-1-d}$ to
$\C\times T^*T^d\times D^2\times\C^{n-1-d}$,
the Niederkr\"uger map from \cite[Proposition 5]{nie06}.


\subsection{The contactisation\label{subsec:contactisation}}

Following the explanations from \cite[Section 2]{gz16b}
we denote the Liouville manifold form
the beginning of Section \ref{sec:stholdiscs}
by $(V,\lambda_V)$
so that its contactisation
$(\R\times V,\rmd b+\lambda_V)$
is equal to the strict contact manifold $(Z,\alpha_Z)$.
The corresponding contact structure $\xi_Z$
is given by the set of tangent vectors
$v-\lambda_V(v)\partial_b$ for all $v\in TV$.


\subsection{Liouville manifold and K\"ahler potential\label{subsec:liouandpot}}

The Liouville manifold $(V,\lambda_V)$
admits a complex structure
\[
J_V:=(-\rmi)\oplus\rmi\oplus\rmi\,,
\]
where $-\rmi$ is meant to be
the negative of the complex structure on $T^d$
obtained by the quotient of $\R^d$ by $2\pi\Z^d$
and $T^*\R^d\equiv\R^{2n}\equiv\C^d$
such that $-\rmi$ is an almost complex structure on $T^*T^d$
compatible with $\rmd\bfp\wedge\rmd\bfq$,
cf.\ \cite[Appendix B]{nie06}.
A strictly plurisubharmonic potential $\psi$
in the sense of \cite[Section 3.1]{gz12}
so that $J_V$ is compatible
with the symplectic form $\rmd\lambda_V$
and $\lambda_V=-\rmd\psi\circ J_V$
is given by
\[
\psi\big(\bfw,z_0,\bfz\big):=
\frac12\sum_{j=1}^dp_j^2+
\frac14|z_0|^2+
\frac12\sum_{j=1}^{n-1-d}y_j^2
\,,
\]
where the point $\bfw\in T^*T^d$
is written in coordinates as
$(q_1,p_1,\ldots,q_d,p_d)$ and
$\bfz=z_1,\ldots,z_{n-1-d}$,
where $z_j=x_j+\rmi y_j$, $j=0,1,\ldots,n-1-d$,
denote coordinates
on $D^2\times\C^{n-1-d}$.


\subsection{The symplectisation\label{subsec:symplectisation}}

For any positive, strictly increasing smooth function
$\tau\equiv\tau(a)$ on $\R$ the symplectisation of $(Z,\alpha_Z)$
is defined to be the symplectic manifold
\[
\big(\R\times Z, \rmd(\tau\alpha_Z)\big)\,.
\]
A compatible and translation invariant
almost complex structure
that preserves the contact hyperplanes $\xi_Z$
on all slices $\{a\}\times Z$
is determined by $\partial_a\mapsto\partial_b$
and the requirement that for all $v\in TV$ the tangent vectors
$v-\lambda_V(v)\partial_b$ get mapped to
$J_Vv-\lambda_V(J_Vv)\partial_b$.
With that choice of an almost complex structure
on $\R\times\R\times V$
the {\bf Niederkr\"uger map}
\[
\Phi(a,b\,;p)=\big(a-\psi(p)+\rmi b,p\big)
\]
is a biholomorphism onto $\C\times V$
equipped with the almost complex structure
$\rmi\oplus J_V$.


\subsection{The Niederkr\"uger transform\label{subsec:niedtrans}}

The resulting holomorphic discs maps
\[
\D\lra
\R\times\R\times T^*T^d\times D^2\times\C^{n-1-d}
\,,
\]
to which we refer as being {\bf standard},
can be parametrised by
\[
u_{\bfs,b}^{\bft,\bfw}(z)=
\Big(
\tfrac14\big(|z|^2-1\big),b\,;\bfw,z,\bfs+\rmi\bft
\Big)
\]
for parameters $b\in\R$, $\bfw\in T^*T^d$,
and $\bfs,\bft\in\R^{n-1-d}$,
cf.\ \cite[Section 2.2]{gz16b}.
Natural Lagrangian boundary conditions for the restrictions
of the standard holomorphic disc maps to $\partial\D$
are given by Lagrangian cylinders
\[
L^{\bft}_{\bfp}:=
\{0\}\times\R\times T^d\times\{\bfp\}\times\partial D^2
\times\R^{n-1-d}\times\{\bft\}
\]
parametrised by $\bft\in\R^{n-1-d}$ and $\bfp\in\R^d$,
which foliate $\{0\}\times\partial Z$.
In order to verify $L^{\bft}_{\bfp}$ to be Lagrangian
observe that the restriction of $\rmd(\tau\alpha_Z)$
to the tangent bundle of $\{0\}\times Z$ equals
$\tau(0)\rmd\alpha_Z$,
which is a positive multiple of
\[
\rmd\bfp\wedge\rmd\bfq
+\rmd x_0\wedge\rmd y_0
+\rmd\bfx\wedge\rmd\bfy
\,,
\]
and that $L^{\bft}_{\bfp}$ is of dimension $n+1$.


\section{A boundary value problem\label{sec:abdryvalprob}}

Let $(W,\omega)$ be the symplectisation
\[
(W,\omega):=\big(\R\times\hat{M}, \rmd(\tau\hat{\alpha})\big)
\]
of the glued
strict contact manifold $(\hat{M},\hat{\alpha})$
introduced in Section \ref{subsec:filling},
where $\tau$ is a positive,
strictly increasing smooth function on $\R$
such that $\tau(a)=\rme^a$ for all $a\geq0$.


\subsection{An almost complex structure\label{subsec:analmcpxstr}}

Let $J$ be a compatible
almost complex structure on $(W,\omega)$
that is invariant under translations in $\R$-direction,
sends the coordinate vector field $\partial_a$
to the Reeb vector field of the contact from $\hat{\alpha}$,
and restricts to a compatible complex bundle structure on
$\big(\hat{\xi},\rmd\hat{\alpha}\big)$,
where $\hat{\xi}$ denotes the contact structure
defined by $\hat{\alpha}$.
Observe that the required conditions for $J$
are satisfied simultaneously for all admissible $\tau$.

We would like to specify a choice
of almost complex structure $J$
in order to deal with the non-compactness of $\hat{M}$.
For positive real numbers $b_0,r,R$
we define the {\bf box} by
\[
B
:=[-b_0,b_0]
\times D_RT^*T^d
\times D^2_r
\times D^{2n-2-2d}_R
\,,
\]
where $D^{2\ell}_{\rho}\subset\C^{\ell}$ denotes
the closed $2\ell$-disc of radius $\rho$
and $D_{\rho}T^*T^d$ is the closed $\rho$-disc
subbundle of $T^*T^d$.
We choose $r<1$ so that the box is contained in $Z$
and require that the interior of the box
contains $\varphi(\partial M)$,
i.e.\
\[
\varphi(\partial M)\subset \Int(B)\subset Z\,.
\]
Similarly to the use of the symbols $\hat{M}$ and $\hat{Z}$
we write $\hat{B}$ for the result of gluing $M$ into $B$.
Observe the chain of strict inclusions
\[
M\subset\hat{B}\subset\hat{Z}\subset\hat{M}
\,.
\]
On the complement of $\R\times\Int(\hat{B})$
we require the almost complex structure $J$
to be the one defined in Section \ref{sec:stholdiscs}
with the obvious modification of the construction
by taking the contactisation
of $\R\times T^*T^d\times\C\times\C^{n-1-d}$
instead of $Z$.
On $\R\times\Int(\hat{B})$ the choice of $J$
will be subject to genericity considerations
specified in Section \ref{subsubsec:lincrop}.


\subsection{The moduli space\label{subsec:themodspace}}

Let $\WW$ be the {\bf moduli space}
of all holomorphic discs
\[
u=(a,f)\co\D\lra (W,J)
\]
for which there exists a {\bf level}
$(\bfp,\bft)\in\R^d\times\R^{n-1-d}$
selecting the Lagrangian boundary cylinder
$L^{\bft}_{\bfp}$ in $\{0\}\times\partial Z$
such that the following boundary condition
is satisfied:
\[
u(\partial\D)\subset L^{\bft}_{\bfp}
\,.
\]
In particular,
the map $\Psi\circ f$,
which is defined in a neighbourhood of $\partial\D$,
is constant along $\partial\D$
setting $\Psi(b,\,.\,)=\psi$ for all $b\in\R$.
Additionally, we require that for all $u\in\WW$
there exist sufficiently large
parameters $b\in\R$, $\bfw=(\bfq,\bfp)\in T^*T^d$,
and $\bfs\in\R^{n-1-d}$ such that the standard disc
$u_{\bfs,b}^{\bft,\bfw}$ -- $(\bfp,\bft)$ being the level of $u$ --
can be regarded as holomorphic disc in $(W,J)$
and is homologous to $u$ in $W$ relative $L^{\bft}_{\bfp}$,
i.e.\
\[
[u]=[u_{\bfs,b}^{\bft,\bfw}]
\qquad\text{in}\quad
H_2(W,L^{\bft}_{\bfp})
\,.
\]
Because of $n\geq2$
all standard holomorphic discs of the same level
are homotopic relative boundary
so that the homological condition is well posed.
The reparametrisation group
-- the group of biholomorphic diffeomorphisms of $\D$ --
is divided out by the requirement
\[
f(\rmi^k)\in
\R\times T^d\times\{\bfp\}\times\{\rmi^k\}
\times\R^{n-1-d}\times\{\bft\}
\qquad\text{for}\quad
k=0,1,2
\,,
\]
i.e.\ $u$ is required to map
the marked points $1,\rmi,-1$
to the characteristic leaves
$L^{\bft}_{\bfp}\cap\{z_0=1\}$,
$L^{\bft}_{\bfp}\cap\{z_0=\rmi\}$,
and $L^{\bft}_{\bfp}\cap\{z_0=-1\}$,
respectively.


\subsection{Convergence\label{subsec:convergence}}

We will study $C^{\infty}$-compactness
properties holomorphic discs in $\WW$.
In the following we list
elementary properties
that all $u=(a,f)\in\WW$ share:


\subsubsection{Uniform energy bounds\label{subsubsec:energybounds}}

 The $L^2$-norm of the gradient is uniformly bounded
 in the sense that the {\bf symplectic energy}
 $\int_{\D}u^*\omega$,
 which is equal to the {\bf action}
 $\int_{\partial\D}f^*\hat{\alpha}$
 of the boundary circle,
 is equal to $\pi$.
 This follows as in \cite[Lemma 3.2]{gz16b} because $u$
 is homologous to a certain standard disc.


\subsubsection{$C^0$-bounds and maximum principle\label{subsubsec:maxprinc}}

 As it is the case for any holomorphic curve
 $u=(a,f)$ in symplectisations
 the function $a$ is subharmonic,
 cf.\ \cite[Lemma 3.6.(i)]{gz16b}.
 In the situation at hand we conclude
 with the arguments from \cite[Lemma 3.6.(i)]{gz16b}
 that $a<0$ on $\Int(\D)$ for all $u=(a,f)\in\WW$.
 
 In order to describe the behaviour
 of $u=(a,f)\in\WW$
 in the direction of $\hat{M}$
 we denote by $G$ the $f$-preimage
 of $\hat{M}\setminus\Int(\hat{B})$.
 Namely,
 on $G$ we can introduce coordinate functions
 \[
 f=(b,\bfw,h_0,\bfh)
\qquad\text{on}\quad
\R\times T^*T^d\times\C\times\C^{n-1-d}
 \]
 according to the indicated splitting.
 By the properties of the Niederkr\"uger map
 the coordinate function $b$ is harmonic
 and the $\bfw,h_0,\bfh$ are holomorphic.
 
 If $G=\D$,
 then $u$ will be one of the standard discs
 $u_{\bfs,b}^{\bft,\bfw}$
 sitting in the complement of $\R\times\Int(\hat{B})$.
 The argument for that is the same as for
 \cite[Lemma 3.7]{gz16b} with the following
 additional observation:
 The holomorphic map $\bfw\co\D\ra T^*T^d$
 lifts to a holomorphic map to the universal cover
 resulting into an anti-holomorphic disc map into $\C^d$
 with boundary circle $\partial\D$
 mapped into a totally real affine plane $\R^d\times\{\bfp\}$.
 Hence, as in \cite[Lemma 3.7]{gz16b}
 or by Schwarz reflection $\bfw$ must be constant.
 An alternative argument is based on the fact
 that the symplectic energy of $\bfw$,
 which is equal to the Dirichlet energy, vanishes,
 so that -- again -- $\bfw$ must be constant.
 Denoting the level of $u$ by $(\bfp,\bft)$
 we use a retraction of $T^*T^d$ to $T^d\times\{\bfp\}$
 to homotope the disc $\bfw$ into
 the Lagrangian submanifold $T^d\times\{\bfp\}$
 relative boundary.
 As the homotoped disc has vanishing symplectic energy
 the symplectic energy of $\bfw$ vanishes
 by Stoke's theorem too.
 
 In the situation that $G$ is a proper subset of $\D$
 we will make the following observations:
 By construction,
 $G$ contains a neighbourhood of $\partial\D$
 so that the strong maximum principle
 and the boundary lemma of E.\ Hopf
 apply to $h_0$.
 Indeed, as in \cite[Lemma 3.6.(ii)]{gz16b}
 we conclude that $f(\Int(\D))$ is contained in $\Int(\hat{Z})$.
 Moreover, by the comments on \cite[p.~669 and p.~671]{gz16b}
 we see that $h_0$ restricts to an immersion on $\partial\D$
 so that $u(\partial\D)$ is positively transverse
 to each of the characteristic leaves
 $L^{\bft}_{\bfp}\cap\{z_0=\rme^{\rmi\theta}\}$,
 $\theta\in[0,2\pi)$, denoting the level of $u$ by $(\bfp,\bft)$.
 By the homological condition
 posed by the boundary value problem for $u=(a,f)\in\WW$
 we infer that $h_0$ restrict in fact to an embedding on $\partial\D$.
 
 Continuing the discussions on the case $G\neq\D$
 we observe that by the arguments in \cite[Lemma 3.8]{gz16b}
 the coordinate function $b$ of $u$,
 as $u$ can not be a standard disc,
 takes values in $[-b_0,b_0]$.
 For that recall that $\Psi\circ f$
 is constant along $\partial\D$.
 Similarly, there exists a real number $R_0>R$
 such that the intersection of the non-standard disc
 $f(\D)$ with
 \[
 \R
 \times\Big(T^*T^d\setminus D_{R_0}T^*T^d\Big)
 \times\C
 \times\Big(\C^{n-1-d}\setminus D^{2n-2-2d}_{R_0}\Big)
 \]
 is empty.
 The projection $\bfh$ to the $\C^{n-1-d}$-factor
 can be treated similarly to \cite[Lemma 3.9]{gz16b}.
 For the cotangent factor notice
 that a composition of local lifts of $\bfw$
 with complex conjugation on the universal cover
 results in local holomorphic maps w.r.t.\
 the standard complex structure on $\C^d$.
 Therefore, the maximum principle implies then
 that $\bfw$ takes values in the codisc bundle
 of radius $|\bfp|$.
 In view of the uniform energy bounds
 stated in Section \ref{subsubsec:energybounds}
 the monotonicity lemma gives an upper bound
 on the level $|\bfp|$ of $u$ analogously
 to the arguments on \cite[p.~674]{gz16b}
 (by possibly using several geodesic balls of the same radius
 less than or equal to $\pi$
 to cut area out of the holomorphic disc $\bfw$).
 
 In conclusion,
 we obtain uniform $C^0$-bounds
 on the $\hat{M}$-part of all non-standard
 holomorphic discs in $\WW$.


\subsubsection{Compactness\label{subsubsec:compactness}}

The space of non-standard holomorphic discs
in $\WW$ is $C^{\infty}$-compact under the assumption
that all contractible periodic Reeb orbits
of $\hat{\alpha}$ have action greater than $\pi$.
This follows with the arguments in \cite[Section 4]{gz16b}
and \cite{fr08,fz15},
and an identification of $T^*T^d$ with $T^d\times\R^d$
so that the variation of the boundary condition $T^d\times\{\bfp\}$
in the cotangent factor of $L^{\bft}_{\bfp}$
can be described with help of translations in $\R^d$.
What remains to show is the following
indecomposability statement:

\begin{lem}
 \label{lem:indecomposable}
 The homology class $[u]$ of all $u\in\WW$
 in $H_2(W,L^{\bft}_{\bfp})$, where $(\bfp,\bft)$
 denotes the level of $u$, is $J$-indecomposable. 
\end{lem}

\begin{proof}
 If not, we could find non-constant holomorphic discs
 $u^1,\ldots,u^N$ with boundary on $L^{\bft}_{\bfp}$
 such that $[u]$ can be decomposed into the sum
 \[
 [u]=m_1[u^1]+\ldots+m_N[u^N]
 \]
 for natural numbers $N$ and $m_1,\ldots,m_N$,
 where at least one of them is greater than $1$.
 By exactness of the symplectic form $\omega$
 none of the holomorphic maps $u^1,\ldots,u^N$
 can be defined on a sphere.
 Denote the restrictions to the boundary by
 $\gamma$ and $\gamma^1,\ldots,\gamma^N$,
 and observe that
 \[
 [\gamma]=m_1[\gamma^1]+\ldots+m_N[\gamma^N]
 \]
 in $H_1(L^{\bft}_{\bfp})$.
 As $L^{\bft}_{\bfp}$ is contained in the complement
 of $\Int(\hat{B})$ the boundary loops admit a splitting
 w.r.t.\
 \[
 \{0\}\times\R\times T^d\times\{\bfp\}\times\partial D^2
 \times\R^{n-1-d}\times\{\bft\}
 \,.
 \]
 The symplectic energy of holomorphic discs
 in $(W,\omega)$ is equal to the action
 of the boundary loops so that by
 positive transversality w.r.t.\ the characteristic leaves
 mentioned in Section \ref{subsubsec:maxprinc}
 we obtain
 \[
 \pi=
 \sum_{j=1}^Nm_j[\gamma^j_{T^d}]+
 \sum_{j=1}^Nm_jn_j\pi
 \]
 for natural numbers $n_1,\ldots,n_N$.
 The first summand equals the total action
 of the projections of $u^j|_{\partial\D}$ to $T^d\times\{\bfp\}$
 and so is in turn equal to the action of the corresponding projection
 $\gamma_{T^d}$ of $\gamma$.
 Because $[u]$ can be represented
 by a standard disc $u_{\bfs,b}^{\bft,\bfw}$, $\bfw=(\bfq,\bfp)$,
 whose boundary map equals
 \[
 \rme^{\rmi\theta}\longmapsto
 \Big(
 0,b\,;\bfw,\rme^{\rmi\theta},\bfs+\rmi\bft
 \Big)\,,
 \]
 the action of $\gamma_{T^d}$ vanishes.
 In total, we reach the inequality $\pi\geq N\pi$.
 Therefore, $N=1$ proving $J$-indecomposability.
\end{proof}


\subsection{Transversality\label{subsec:transversality}}

Each standard holomorphic disc in $\WW$
admits a neighbourhood that can be parametrised
by
\[
\Big(
b\,;\bfw,\bfs+\rmi\bft
\Big)
\in
\R\times T^*T^d\times\C^{n-1-d}
\,.
\]
We will show that a similar parametrisation
near each of the non-standard holomorphic discs
in $\WW$ exists so that $\WW$
will be a smooth manifold of dimension $2n-1$.


\subsubsection{Maslov index\label{subsubsec:maslovindex2}}

The Maslov index of all $u\in\WW$ is equal to $2$.
By the considerations in \cite[Lemma 3.1]{gz16b}
it is enough to compute
the Maslov index for all standard discs $u\in\WW$,
which lift up to standard holomorphic discs in
\[
\R\times\R\times T^*\R^d\times D^2\times\C^{n-1-d}
\equiv
\R\times\R\times D^2\times\C^{n-1}
\]
in the sense of \cite{gz16b}.


\subsubsection{Simplicity\label{subsubsec:simplicity}}

Using Lemma \ref{lem:indecomposable}
and \cite[Theorem A]{lazz11}
we see that all holomorphic discs
$u\in\WW$ are simple,
cf.\ \cite[Lemma 3.4]{gz16b}.
Based on that one shows as in 
\cite[Lemma 3.5]{gz16b}
that for all $u=(a,f)\in\WW$
the set of all $f$-injective points
is open and dense in $\D$.
One only has to observe
that the projection $h_0$ of $f$
to the $D^2$-factor is an embedding
along $\partial\D$ so that $u|_{\partial\D}$
is positively transverse to the characteristic leaves
-- as done in Section \ref{subsubsec:maxprinc}.


\subsubsection{Linearised Cauchy--Riemann operator\label{subsubsec:lincrop}}

Based on Section \ref{subsubsec:simplicity}
one chooses a regular almost complex structure $J$
as in \cite[Section 5.2]{gz16b}
so that the linearised Cauchy--Riemann operator $D_u$
is onto for all $u\in\WW$.
Taking variations of the level parameters
$(\bfp,\bft)\in\R^d\times\R^{n-1-d}$
for the Lagrangian boundary conditions
$L^{\bft}_{\bfp}$ induced by translations
similarly to \cite[Section 4.1]{gz16b}
one computes
using Section \ref{subsubsec:maslovindex2}
the Fredholm index of $D_u$
to be $2n+2$ as in \cite[Section 5.1]{gz16b}.
Subtracting $3$ for the marked points
fixed by three characteristic leaves yields $2n-1$,
which turns out to be the dimension of $\WW$.
In fact, $\WW$ is a smooth manifold
that admits a natural orientation
obtained by the orientation of $D_u$
described in \cite[Section 5.3]{gz16b}
-- observe that $L^{\bft}_{\bfp}$
admits a canonical parallelization and
and the above mentioned variations
of the Lagrangian boundary conditions $L^{\bft}_{\bfp}$.


\section{The homotopy type\label{sec:thehomotopytype}}

In Section \ref{sec:abdryvalprob} we showed
that the moduli space $\WW$ is a smooth,
naturally oriented manifold of dimension $2n-1$
such that the evaluation map
\[
 \begin{array}{rccc}
  \ev\co & \WW\times\D     & \longrightarrow & \hat{Z}\\
            & \bigl( (a,f),z\bigr) & \longmapsto     & f(z)
 \end{array}
\]
is proper and of degree one.
In this section we will use the evaluation map
in order to draw conclusions
on the homotopy type of the contact manifold $M$.


\subsection{Homology type and fundamental group\label{subsec:homologytypefundamantal}}

With the argumentation used in
\cite[Sections 2.3 and 2.5]{bgz}
and \cite[Section 6]{gz16b}
we obtain that the evaluation map
\[
\ev\co\WW\times\D\longrightarrow\hat{Z}
\]
is surjective in homology and $\pi_1$-surjective.
The restriction to $z_0=1$ is given by
\[
\ev\co\WW\times\{1\}
\longrightarrow
\bigcup_{(\bfp,\bft)\in\R^d\times\R^{n-1-d}}
L^{\bft}_{\bfp}\cap\{z_0=1\}
\]
the target being equal to
\[
\R\times T^*T^d\times\{1\}\times\C^{n-1-d}
\,.
\]
Both evaluation maps
complete to a commutative square
\begin{diagram}
\WW \times \D & 
	\rTo^{\qquad\qquad\ev} & 
	\hat{Z} \\
\uTo_\subset &
	&
	\uTo_\subset \\
\WW\times\{1\} & 
	\rTo^{\quad\quad\ev\quad\quad} &
	\R\times T^*T^d\times\{1\}\times\C^{n-1-d} \\
\end{diagram}
via the homotopy equivalence
\[
\WW\times\{1\}
\subset
\WW\times\D
\]
and the inclusion
\[
\R\times T^*T^d\times\{1\}\times\C^{n-1-d}
\subset\hat{Z}
\,.
\]
Therefore, the induced map
\[
T^d\subset\hat{Z}\simeq M
\]
is surjective in homology and $\pi_1$-surjective too.

Similarly to \cite[Section 2.4]{bgz} one shows
-- by simply replacing $2n$ by $2n+1$ --
that $H_kM=0$ for all higher degrees
$k\geq d+1$ and that the inclusion of
\[
\partial M\lra M
\]
induces an isomorphism for the homology groups $H_k$
of low degree
\[
k=0,1,\ldots,2n-1-d
\,.
\]
Therefore,
\[
H_*M=H_*\Big(T^d\times D^{2n+1-d}\Big)
\,.
\]
Because the fundamental group of $\partial M$
is abelian we infer with \cite[Section 2.5]{bgz}
that the inclusion $\partial M\subset M$
is $\pi_1$-isomorphic.
In particular,
\[
\pi_1M=\pi_1\Big(T^d\times D^{2n+1-d}\Big)
\,.
\]


\subsection{A cobordism\label{subsec:acobordism}}

Recall that $\hat{Z}$
is obtained by removing the bounded component
of the complement of $\varphi(\partial M)\subset\Int(Z)$ in $Z$
and gluing with $M$ along the boundaries via $\varphi$.
By assumption,
$\varphi(\partial M)$ is isotopic to the sphere bundle
$S\big(T^*T^d\oplus\underline{\R}^{2n+1-2d}\big)$
inside $Z$ viewed as a subset of
$\R\times T^*T^d\times D^2\times\C^{n-1-d}$.
Choosing a suitable bundle metric for
$T^*T^d\oplus\underline{\R}^{2n+1-2d}$
we assume that $\varphi(\partial M)$ is contained
in the interior of the corresponding disc bundle
$D\big(T^*T^d\oplus\underline{\R}^{2n+1-2d}\big)$.
After gluing $M$ to the unbounded component
in the total space we obtain 
a manifold denoted by $M_1$,
which deformation retracts onto $M$
and is homotopy equivalent to $\hat{Z}$.
Furthermore,
we denote by $M_0$
a possibly rescaled copy of 
$D\big(T^*T^d\oplus\underline{\R}^{2n+1-2d}\big)$
to which a nowhere vanishing section is added
-- before gluing --
so that $M_0$ is contained in $M_1\setminus M$ -- after gluing.
In the following we will study
homotopical properties
of the cobordism
\[
X:=M_1\setminus\Int M_0
\,.
\]

Denote by $D_0\subset\partial M_0$
and $D_1\subset\partial M_1$
in $M_1$ isotopic copies of the disc bundle of
$\R\times T^*T^d\times\{1\}\times\C^{n-1-d}$,
which strongly deformation retracts to $T^d$.
The isotopy between $D_0$ and $D_1$
can be chosen to be the restriction of
the above used shift and rescaling
isotopy between $\partial M_0$ and $\partial M_1$,
and extends to an isotopy between
$M_0$ and $M_1$.
Therefore, we obtain the following
homotopy commutative diagram:
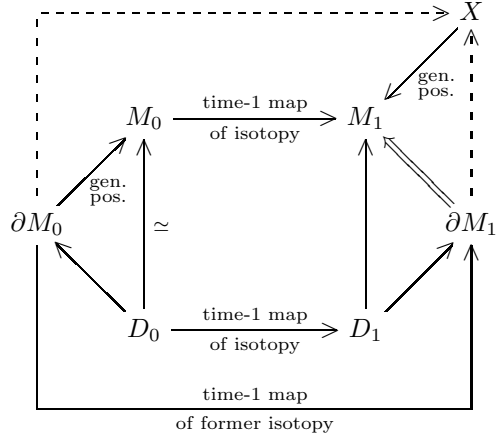
\begin{diagram}
\HmeetV  &
	\rDashto &
	&
	&
	&
	&
	X \\
&
	&
	&
	&
	&
	\ldTo_{\substack{\textrm{gen.}\\\textrm{pos.}}} &
	\\
&
	&
	M_0 &
	\rTo^{\textrm{time-1 map}}_{\textrm{of isotopy}} &
	M_1 &
	&
	\\
\uDash&
	\ruTo_{\substack{\textrm{gen.}\\\textrm{pos.}}} &
	&
	&
	&
	\luDoubleto &
	\uDashto\\
\partial M_0 &
	&
	\uTo_{\simeq} &
	&
	&
	&
	\partial M_1 \\
\dLine &
	\luTo &
	 &
	&
	\uTo &
	\ruTo&
	\\
&
	&
	D_0 &
	\rTo^{\textrm{time-1 map}}_{\textrm{of isotopy}} &
	D_1 &
	&
	\uTo \\
\HmeetV &
	&
	&
	\rLine^{\substack{\\ \\ \\ \textrm{time-1 map}}}_{\textrm{of former isotopy}} &
	&
	&
	\HmeetV \\
\end{diagram}
where all indicated maps are obtained by inclusion
with the exception of the diffeomorphisms
$D_0\ra D_1$ and $\partial M_0\ra\partial M_1$.
The arrow $M_0\ra M_1$ can be alternatively understood
to mean the time-$1$ map of the described isotopy
besides the meaning of the inclusion.

As in \cite[Lemma 5.1 and 5.2]{bgz} one shows
that the inclusions
\[
\partial M_0, \partial M_1\lra X
\]
are $\pi_1$- and $H_*$-isomorphic.
Indeed,
in low degrees $k=0,1,\ldots,2n-1-d$
this follows with 
the results stated in
Section \ref{subsec:homologytypefundamantal}
and
general position arguments,
which are available whenever $k+d<2n+1$,
using $n>d$ and $n\geq 2$;
the arrows labeled by {\it gen.\ pos.} are
homotopy (resp. homology) isomorphisms in those degrees
in view of the induced long exact sequences.
In higher degrees $k\geq d+1$,
in which the homology groups of $M_0$ and $M_1$
vanish by Section \ref{subsec:homologytypefundamantal},
this follows with the induced long exact sequence
of the pair $(M_1,M_0)$, excision and
Poincar\'e duality
applied to the compact cobordism $X$
in combination with the universal coefficient theorem.

\begin{rem}
 \label{rem:isincliso?}
 The above arguments show
 that the inclusion of $M_0$ into $M_1$
 induces an isomorphism in homology.
 This is {\it a priori} not clear even if the involved
 homology groups are isomorphic.
\end{rem}


\subsection{Being an $h$-cobordism\label{subsec:bnhcob}}

In order to prove that the inclusions
$\partial M_0, \partial M_1\ra X$
are in fact homotopy equivalences
one shows that the topological pairs
$(X,\partial M_0)$ and $(X,\partial M_1)$
are homotopically trivial.
As these are homotopically trivial
an application of the relative Hurewicz theorem
shows that the quotients of
the relative homotopy groups
by the action of the fundamental group of $\partial M$,
which is isomorphic to $\Z^d$,
are trivial.

To conclude with the vanishing of the relative homotopy groups 
we will employ universal coverings as done in the argumentation
in \cite[Section 6]{bgz}.
The following proposition is based on the triviality
of $\pi_1(X,\partial M_0)$ and $\pi_1(X,\partial M_1)$
obtained in Section \ref{subsec:acobordism}:

\begin{prop}
\label{prop:xinanhcob}
 $X$ is an $h$-cobordism.
\end{prop}

In order to prove Proposition \ref{prop:xinanhcob}
we consider the universal covering
$\pi\co\widetilde{M}\ra M$ provided with the
contact form $\tilde{\alpha}:=\pi^*\alpha$,
which satisfies $\inf_0(\tilde{\alpha})>\pi$.
Observe that the covering $\pi$ is infinite
as $\pi_1M=\Z^d$, see
Section \ref{subsec:homologytypefundamantal}.
The induced covering
on the boundary
$\pi|_{\partial\widetilde{M}}\co\partial\widetilde{M}\ra\partial M$
is given by the universal covering
$\R^d\times S^{2n-d}\ra T^d\times S^{2n-d}$
caused by the $\pi_1$-isomorphicity
of the inclusion $\partial M\subset M$, see again
Section \ref{subsec:homologytypefundamantal}.
Furthermore,
the universal covering space of $\hat{Z}$
is made up by the analogue
of the gluing construction of $\hat{Z}$
involving this time
$\widetilde{M}$ and
\[
\widetilde{Z}=
\R\times T^*\R^d\times D^2\times\C^{n-1-d}=
D^2\times\R^{2n-1}
\]
with the gluing map being a lift of $\varphi\circ\pi$.
The universal covering map of $\hat{Z}$
restricts to $\pi$ on $\widetilde{M}$.

Similarly to Section \ref{sec:abdryvalprob}
and \cite[Section 6]{bgz}
one defines a moduli space $\WW'$
of holomorphic discs
in the covering space of $\hat{Z}$
with respect to the lift
of the almost complex structure $J$.
This results into a covering
$\WW'\ra\WW$ of moduli spaces.
As in \cite[Lemma 6.1]{bgz} one shows
that the evaluation map
\[
\ev\co\WW'\times\D\longrightarrow\widetilde{\hat{Z}}
\]
is proper of degree $1$
because the projections of all holomorphic discs
in $\WW'$ to the $\R$-factor of the symplectisation
are contained in a uniform compact interval.
Alternatively, one can compensate
the non-compactness caused by
$\widetilde{M}$
with the results in \cite{bwz}
applied to the {\it trivial} virtually contact structure
given by the universal covering of $\hat{Z}$.
Here, by a {\bf virtually contact structure} we mean
a Riemannian covering together with a contact form primitive
of the pull back of an odd-symplectic form on the base
that is uniformly bounded from below and above.
The virtually contact structure is {\bf trivial}
if the contact form is obtained from a primitive
of the odd-symplectic from on the base by pull back.

As in Section \ref{subsec:homologytypefundamantal}
one considers a diagram
\begin{diagram}
\WW ' \times \D & 
	\rTo^{\qquad\qquad\ev} & 
	\widetilde{\hat{Z}} \\
\uTo_\subset &
	&
	\uTo_\subset \\
\WW ' \times\{1\} & 
	\rTo^{\quad\quad\ev\quad\quad} &
	\R\times T^*\R^d\times\{1\}\times\C^{n-1-d} \\
\end{diagram}
and concludes that $\widetilde{M}$ is contractible,
cf.\ \cite[Proposition 6.2]{bgz}
and its preceding remarks in \cite{bgz}.

\begin{proof}[{\bf Proof of Proposition \ref{prop:xinanhcob}}]
  With the above shown contractibility of $\widetilde{M}$
  the claim follows with the purely topological
  argumentation used in \cite[Theorem 9.1]{bgz}.
  Alternatively, one can follow the reasoning
  in \cite[Section 8]{bgz} or \cite[Section 2.5]{kbdiss}
  invoking simplicity of the topological space $\partial M$
  and the cobordism diagram from
  Section \ref{subsec:acobordism} similarly to \cite[Lemma 6.3]{bgz}.
\end{proof}

\begin{proof}[{\bf Proof of Theorem \ref{thm:mainthm}}]
 With Proposition \ref{prop:xinanhcob}
 the claim follows with standard arguments
 based on Whitehead's theorem, $\mathrm{Wh}(\Z^d)=0$,
 and the $s$-cobordism theorem
 as done in \cite{kbdiss,bgz}.
\end{proof}


\section{Appendix: Local contact inversion \label{app:loconinv}}

As in knot theory we call
the complement of the interior of a tubular neighbourhood
of a submanifold $Q$ the {\bf exterior} of $Q$.
We show that a collar extension of the exterior
of a subcritically isotropic torus in a contact manifold
admits a positive contact inversion
along the boundary of the exterior.
This allows surgerial constructions for contact manifolds
near subcritically isotropic tori
similar to the considerations in Section \ref{subsec:filling}.


\subsection{Hypersurfaces transverse to a Liouville flow\label{subsec:hyptrnliouvecfield}}

Let $(V,\lambda)$ be a Liouville manifold
with symplectic form $\omega=\rmd\lambda$
and Liouville vector field $Y$ defined by
$i_Y\omega=\lambda$.
Let $M_0$ and $M_1$ be hypersurfaces in $V$
that are transverse to $Y$
such that $\alpha_i:=\lambda|_{TM_i}$, $i=0,1$,
is a contact form.
We assume that
each flow line of $Y$
intersects each of $M_0$ and $M_1$ in a single point
so that the resulting correspondence
forms a bijection $M_0\ra M_1$.

Under the stated assumptions
we find a domain $D$ in $\R\times M_0$
that contains $\{0\}\times M_0$
such that for all $p\in M_0$
the intersection of $D$ with each line $\R\times\{p\}$
is the maximal interval on which the flow line
$t\mapsto\varphi_t(p)$ of $Y$ is defined.
This defines an
embedding $\Phi\co D\ra V$ of Liouville manifolds
(i.e.\ $\Phi^*\lambda=\rme^t\alpha_0$)
via $\Phi(t,p)=\varphi_t(p)$ for all $(t,p)\in D$
such that $M_1\subset\Phi(D)$
is the $\Phi$-image of the graph of a smooth function $f\co M_0\ra\R$
and $\alpha_1$ corresponds to the contact form
$\rme^f\alpha_0$ on $\Phi^{-1}(M_1)$.
Setting $\psi(0,p)=\big(f(p),p\big)$ for all $p\in M_0$
we obtain a strict contactomorphism
$\Phi\circ\psi\co(M_0,\rme^f\!\alpha_0)\ra(M_1,\alpha_1)$.


\subsection{A model involution\label{subsec:amodleinvolution}}

We consider the Liouville manifold
\[
\Big(T^*T^d\times\C^{n+1-d},
\bfp\!\;\rmd\bfq
+\frac12\big(\bfx\rmd\bfy-\bfy\rmd\bfx\big)\Big)
\]
with symplectic form
\[
\rmd\bfp\wedge\rmd\bfq
+\rmd\bfx\wedge\rmd\bfy
\]
and Liouville vector field
\[
Y_0=\bfp\!\;\partial_{\bfp}
+\frac12\big(\bfx\partial_{\bfx}+\bfy\partial_{\bfy}\big)\,,
\]
which is transverse to
\[
M=\Big\{|\bfp|^2+|\bfx|^2+|\bfy|^2=1\Big\}
\]
defining a contact form $\alpha_0$ on $M$.
The involution
\[
\iota(\bfq,\bfp;\bfx,\bfy)=(\bfq,\bfp;-x_1,x_2,\ldots,x_{n+1-d},-y_1,y_2,\ldots,y_{n+1-d})
\]
preserves the Liouville form
and the $M$-defining distance function
inducing a strict contactomorphism of $(M,\alpha_0)$.
Observe that $\iota$ interchanges
\[
M\cap\big\{\!\pm\!x_1\geq0\big\}\cong T^d\times D^{2n+1-d}
\]
as well as the isotropic tori
\[
T_{\pm}=\big\{(\bfq,{\mathbf 0};\pm\!1,0,\ldots,0)\,\big|\,\bfq\in T^d\big\}\,.
\]

We remark
that the (by a Hamiltonian vector field)
shifted Liouville vector field
\[
Y_1=Y_0+\frac14\partial_{x_1}
\]
defines the standard (rotationally invariant) contact form
\[
\alpha_1=\rmd t+\bfp\!\;\rmd\bfq
+\frac12\big(\bfx\rmd\bfy-\bfy\rmd\bfx\big)
\]
on
\[
M_1=T^*T^d\times\left\{x_1=\frac23\right\}\times\R_{y_1}\times\C^{n-d}
\equiv\R_t\times T^*T^d\times\C^{n-d}\,,
\]
where $y_1$ is renamed in $t$ and
where now $\bfx$ and $\bfy$
stand for the corresponding tuples
with $x_1$ and $y_1$ deleted.
With \cite[Example 2.1.3]{gei08}
$\alpha_1$ can be brought to $\alpha_Z$
by a strict contactomorphism.


\subsection{Interpolating Liouville vector fields\label{subsec:intliouvec}}

We continue the discussion from Section
\ref{subsec:amodleinvolution}.
Let $\chi\equiv\chi(\bfq,\bfp;\bfx,\bfy)$ be the cut off function
$\tilde{\chi}\big(|\bfp|^2+|\bfx|^2+|\bfy|^2\big)$
induced by a smooth function $\tilde{\chi}\co[0,\infty)\ra[0,1/4]$
that is equal to $0$ on $[0,1]$, strictly increasing on $(1,3/2)$,
and equal to $1/4$ on $[3/2,\infty)$.
With respect to the Hamiltonian function $H=-\chi y_1$
we define the Liouville vector field $Y=Y_0+X_H$,
where
\[
X_H=
2\tilde{\chi}'y_1
\big(
-\bfp\!\;\partial_{\bfq}+\bfy\partial_{\bfx}-\bfx\partial_{\bfy}
\big)
+\chi\partial_{x_1}
\]
denotes the Hamiltonian vector field of $H$.
The Liouville vector field $Y$ equals $Y_0$
on $\big\{|\bfp|^2+|\bfx|^2+|\bfy|^2\leq1\big\}$
and $Y_1$ on
$\big\{|\bfp|^2+|\bfx|^2+|\bfy|^2\geq3/2\big\}$;
each flow line of $Y$ connects $M$ with
$\big\{|\bfp|^2+|\bfx|^2+|\bfy|^2=3/2\big\}$
as
\[
\rmd\big(|\bfp|^2+|\bfx|^2+|\bfy|^2\big)(Y)
=|\bfp|^2+\frac12\big(|\bfx|^2+|\bfy|^2\big)+\chi x_1
\]
defining a bijection between the two hypersurfaces.
Denote by $M_0$ the set of intersection points
of $M$ with those flow lines of $Y$ that intersect
$\big\{|\bfp|^2+|\bfx|^2+|\bfy|^2=3/2\big\}$
along $\big\{x_1>-1/2\big\}$.
Observe that $M_0$ is an open neighbourhood of
$M\cap\{x_1=0\}$ in $M$
as
\[\rmd x_1(Y)=\frac12x_1+2\tilde{\chi}'y_1^2+\chi\]
is positive on $\big\{1<|\bfp|^2+|\bfx|^2+|\bfy|^2\big\}$
along $\{x_1=0\}$.
With the considerations from Section
\ref{subsec:amodleinvolution}
we obtain a strict contactomorphism
$\varphi\co(M_0,\rme^f\!\alpha_0)\ra(M_1,\alpha_1)$
defined in a neighbourhood of the invariant set $M\cap\{x_1=0\}$
of the involution $\iota$.


\subsection{Local inversion\label{subsec:locinv}}

Continuing the discussion from Section
\ref{subsec:intliouvec} we observe
that $\varphi\circ\iota\circ\varphi^{-1}$ defines
a strict contactomorphism of $\alpha_1$ on
$\varphi\big(M_0\cap\iota(M_0)\big)$
that leaves $\varphi\big(M_0\cap\{x_1=0\}\big)$
invariant changing (co-)orientations.
We call such a positive (not necessarily strict)
contactomorphism a {\bf contact inversion}
along $\varphi\big(M_0\cap\{x_1=0\}\big)$.

Observe that $\varphi\big(M_0\cap\{x_1=0\}\big)$
as boundary of $\varphi\big(M_0\cap\{x_1\geq0\}\big)$
can be brought into any neighbourhood
of the zero section
$\varphi\big(M_0\cap T_+\big)$ of
$T^*T^d\oplus\underline{\R}^{2n+1-2d}$
using the contact vector field
\[
X=t\partial_t+\bfp\!\;\partial_{\bfp}
+\frac12\big(\bfx\partial_{\bfx}+\bfy\partial_{\bfy}\big)
\]
on $\big(\R\times T^*T^d\times\C^{n-d},\alpha_1\big)$.
Conjugating the contact inversion
$\varphi\circ\iota\circ\varphi^{-1}$
with the flow of $X$ we obtain combined with 
the isotropic neighbourhood theorem:

\begin{prop}
\label{prop:contact-inversion}
 Any isotropic submanifold
 of a given contact manifold,
 whose conformally symplectic normal bundle
 possesses a non-vanishing section,
 admits a tubular neighbourhood $U$ together with a
 contact inversion along the boundary $\partial U$.
\end{prop}

Indeed, this is because
the submanifold being a subcritically isotropic torus
with trivial conformally symplectic normal bundle
is not (really) used in the above construction.


\begin{ack}
  This work could not ever been finished without the
  immense support of
  Peter Albers and Hansj\"org Geiges.
  We would like to thank Joel Fish
  for drawing our attention to {\it symplectic dynamics}
  and Marc Kegel for showing us the beauty of knot complements.
\end{ack}


\end{document}